\documentclass[11pt,reqno]{amsproc}
\usepackage[all]{xy}
\usepackage[colorlinks,urlcolor=blue,citecolor=blue,linkcolor=blue]{hyperref}
\usepackage{natbib,xcolor,graphicx,stix2,bm,anysize}
\usepackage{caption,subcaption,eqnarray}

\marginsize{15mm}{15mm}{6mm}{5mm}

\theoremstyle{theorem}
\newtheorem{theorem}{Theorem}[section]
\newtheorem{proposition}[theorem]{Proposition}
\newtheorem{lemma}[theorem]{Lemma}

\theoremstyle{remark}

\newtheorem{remark}[theorem]{\it Remark}

\theoremstyle{definition}

\AtBeginDocument{.}

\numberwithin{equation}{section}
\allowdisplaybreaks

\newcommand{\tmop}[1]{\ensuremath{\operatorname{#1}}}

\begin{document}
\title[biharmonic equations with $p$-Laplacian]{On biharmonic equations with $p$-Laplacian and indefinite potentials or critical nonlinearity}
\author[S. Liu \& K. Perera]{Shibo Liu\quad Kanishka Perera\vspace{-1em}}
\dedicatory{Department of Mathematics and Systems Engineering, Florida Institute of Technology\\
Melbourne, FL 32901, USA}
\thanks{Emails: \texttt{\bfseries sliu@fit.edu} (S. Liu)}
\begin{abstract}
In this paper we consider nonlinear biharmonic equations with $p$-Laplacian ($p\ge2$) of the form
  $$
  \left\{ \begin{array}{l}
    \Delta^2 u - \Delta_p u + V (x) u = f (x, u) \text{,}\\
    u \in H^2 (\mathbb{R}^N) \text{,}
  \end{array} \right.
  $$
  where the potential $V(x)$ may be indefinite. Using local linking and Morse theory, nontrivial solutions are obtained. In case the nonlinearity $f(x,\cdot)$ is odd, we obtain a sequence of large energy solutions. In the second part of the paper, for bounded positive potential, we get multiple solutions for the case that $$f(x,u)=\lambda g (x) | u |^{q - 2} u + | u |^{m
  - 2} u$$
  with exponent $m$ critical or subcritical.
\end{abstract}
\maketitle

\section{Introduction}

The biharmonic operator $\Delta^2$ appears in many problems in physics and
engineering such as the nonlinear oscillation in a suspension bridge
{\cite{MR1084570}} and the static deflection of an elastic plate in a fluid
{\cite{Abrahams2002}}. In the last two decades elliptic equations
\begin{equation}
  \left\{ \begin{array}{l}
    \Delta^2 u - \Delta u + \lambda V (x) u = f (x, u) \text{,}\\
    u \in H^2 (\mathbb{R}^N) \text{}
  \end{array} \right. \label{1e}
\end{equation}
involving the biharmonic operator have captured great research interests. Such
equations were studied in {\cite{MR2735556,MR2927503}} for the case that the
Lebesgue measure $\mathfrak{m} (\{ V \leq b \}) < \infty$ for all $b > 0$, and
in {\cite{MR2948252,MR3062426}} for the case $\mathfrak{m} (\{ V \leq b \}) <
\infty$ for some $b > 0$. In the latter case, solutions are obtained for
$\lambda \gg 1$.

{\cite{MR3723523}} considered (\ref{1e}) with a
sublinear term adding to the nonlinearity:
\[ f (x, u) = g (x, u) + \mu \xi (x) | u |^{q - 2} u \]
for some $q \in (1, 2)$. The fact that the sublinear term prevents the zero
function $u = 0$ to be a local minimizer of $\Phi$ enables the authors to get two
nontrivial solutions via the Ekeland variational principle {\cite{MR346619}}
and the mountain pass theorem {\cite{MR0370183}}. In
{\cite{MR4154811,MR4066583}}, this result was extended to more general
equations
\begin{equation}
  \left\{ \begin{array}{l}
    \Delta^2 u - \Delta_p u + V (x) u = f (x, u) \text{,}\\
    u \in H^2 (\mathbb{R}^N) \text{,}
  \end{array} \right. \label{e1}
\end{equation}
where $p \in [2, 2^{\ast}]$ being $2^{\ast} = \frac{2 N}{N - 2}$ for $N \geq
3$ and $2^{\ast} = \infty$ for $N \in \{ 1, 2 \}$. Along this line, when $g
(x, \cdot)$ is odd, the existence of infinitely many solutions was obtained
in {\cite{MR4597960}}. See {\cite{MR4342976,MR4385750}} for more recent
results on the problem (\ref{1e}) with critical nonlinearity $f (x, t) \sim |
t |^{2_{\ast} - 2} t$ as $| t | \rightarrow \infty$, where the critical
Sobolev exponent $2_{\ast} = \frac{2 N}{N - 4}$ for $N \geq 5$ and $2_{\ast} =\infty$ for $N \leq 4$. See also {\cite{MR4547556}}
for recent results on (\ref{1e}) with general nonlinearity in the spirit of
{\cite{MR695535,MR695536}}.

{\cite{MR3569412}}, initiated the study of biharmonic
equations with $p$-Lap\-l\-a\-c\-i\-a\-n of the form (\ref{e1}). Since then, elliptic
equations of the form (\ref{e1}) have been extensively studied. In addition to
{\cite{MR4597960,MR4154811,MR4066583}} mentioned above, the reader is referred
to {\cite{MR4901051,MR4695885,MR4741178}} for the case that the nonlinearity
grows critically, {\cite{MR4681495}} for positive or zero mass cases in the
sense of {\cite{MR695535,MR695536}}. See also {\cite{MR5011841}} for results
on normalized solutions (solutions with prescribed $L^2$-norm).

In all these results on (\ref{e1}) the operator $\mathcal{H}= \Delta^2 + V
(x)$ is positive definite and the corresponding quadratic form $Q$ given in
(\ref{Bb}) below defines a norm. As a consequence, one can apply the mountain pass
theorem to find solutions at positive energy level. The purpose of this paper
is to study the case that $\mathcal{H}$ may not be positive definite, this
happens if the ``size'' of $V_- = \min \{ V, 0 \}$ is large.

Suppose $V \in C (\mathbb{R}^N)$ is {\emph{bounded below}}. Take $m > 1$ such
that $\tilde{V} = V + m \geq 1$. On the subspace
\[ X = \left\{ u \in H^2 \left| \int V (x) u^2 < \infty \right. \right\} \]
of $H^2 = H^2 (\mathbb{R}^N)$, we equip the norm
\[ \| u \|_V = \left( \int (| \Delta u |^2 + \tilde{V} (x) u^2) \right)^{1 /
   2} \text{.} \]
As in {\citet[{\textsection}2]{MR3569412}}, by Gagliardo--Nirenberg inequality
{\cite{MR109940}} there is $A_0 > 0$ such that
\begin{align}
  \int | \nabla u |^2 & \leq  A_0^2 \left( \int | \Delta u |^2 \right)^{1 /
  2} \left( \int u^2 \right)^{1 / 2} \nonumber\\
  & \leq  A_0^2 \left( \int | \Delta u |^2 \right)^{1 / 2} \left( \int
  \tilde{V} (x) u^2 \right)^{1 / 2} \nonumber\\
  & \leq  \frac{A_0^2}{2} \int (| \Delta u |^2 + \tilde{V} (x) u^2) \text{.}
  \label{0} 
\end{align}
Consequently, $X$ is a Hilbert space continuously embedded into $H^2
(\mathbb{R}^N)$.

\begin{remark}
  \label{r1}From (\ref{0}), it is clear that: (a) If $V$ is {\emph{bounded}},
  then $\| \cdot \|_V$ is equivalent to the standard $H^2$-norm and our
  Hilbert space $X$ is precisely $H^2 (\mathbb{R}^N)$.
  
  (b) If $V$ is coercive, then equipped with the norm
  \[ \| u \|_{\#} = \left( \int (| \Delta u |^2 + \tilde{V} (x) u^2 + | \nabla
     u |^2) \right)^{1 / 2} \]
  $X$ can be compactly embedded into $L^2 = L^2 (\mathbb{R}^N)$, see
  {\citet[Lem 2.1]{MR3276713}}. By (\ref{0}), the two norms $\| \cdot \|_V$
  and $\| \cdot \|_{\#}$ are equivalent. Hence, with our norm $\| \cdot
  \|_V$ or its equivalent norm $\| \cdot \|$ defined at the beginning of
  {\textsection}2, we have compact embedding $X \hookrightarrow L^m$ for $m
  \in [2, 2_{\ast})$.
\end{remark}

To state our results, we need the following assumptions on the potential $V
(x)$ and the nonlinearity $f (x, u)$.
\begin{enumerate}
  \item[$(V_0)$] $V \in C (\mathbb{R}^N)$ is bounded below such that the
  quadratic form $Q : X \rightarrow \mathbb{R}$,
  \begin{equation}
    Q (u) = \frac{1}{2} \int (| \Delta u |^2 + V (x) u^2) \label{Bb}
  \end{equation}
  is nondegenerate (which means that the self-adjoint operator associated to $Q$ is an isomorphism), and the negative space of $Q$ is finite dimensional.
  
  \item[$(f_{\ast})$] $f \in C (\mathbb{R}^N \times \mathbb{R})$ is
  subcritical: there are $C > 0$ and $q \in (p, 2_{\ast})$ such that
  \[ | f (x, t) | \leq C (| t | + | t |^{q - 1}) \text{\qquad for $(x, t) \in
     \mathbb{R}^N \times \mathbb{R}$.} \]
  \item[$(f_0)$] As $t \rightarrow 0,$ $f (x, t) = o (t)$ uniformly on $x \in
  \mathbb{R}^N$.
  
  \item[$(f_1)$] There is $\mu > p$ such that
  \[ 0 < F (x, t) := \int_0^t f (x, \cdot) \leq \frac{1}{\mu} f (x, t)
     t \text{\quad for $x \in \mathbb{R}^N$, $t \neq 0$.} \]
  \item[$(f_2)$] For every $r > 0$, we have
  \[ \lim_{| x | \rightarrow \infty} \sup_{| t | \leq r} \left| \frac{f (x,
     t)}{t} \right| = 0 \text{.} \]
\end{enumerate}
We remark that the assumption $(V_0)$ on the potential $V$ is quite general.
For example, suppose
\[ 0 < b := \lim_{| x | \rightarrow \infty} V (x) \leq + \infty \text{.}
\]
Similar to {\citet[Thms 8.2.1, 8.3.2]{MR2365191}}, the operator $\mathcal{H}=
\Delta^2 + V (x)$ is self-adjoint and bounded below. The spectrum of
$\mathcal{H}$ below $b$ consists of eigenvalues $\lambda_1 \leq \lambda_2 \leq
\cdots$ of finite multiplicity, which can only accumulate at $b$. Therefore, if $0 \in
(\lambda_{\ell}, \lambda_{\ell + 1})$ then $Q$ given by (\ref{Bb}) is
nondegenerate with finite dimensional negative space spanned by eigenfunctions
corresponding to $\{ \lambda_i \}_{i = 1}^{\ell}$.


\begin{theorem}
  \label{t1}Suppose $(V_0)$, $(f_{\ast})$, $(f_0)$, $(f_1)$ and $(f_2)$ are
  satisfied with $p\in(2,2^*]$, then \eqref{e1} has a nontrivial solution. If moreover $f (x,
  \cdot)$ is odd, then \eqref{e1} has a sequence of solutions $\{ u_n \}$
  such that
  \begin{equation}
    \frac{1}{2} \int (| \Delta u_n |^2 + V (x) u_n^2) + \frac{1}{p} \int |
    \nabla u_n |^p - \int F (x, u_n) \text{} \rightarrow + \infty \text{.}
    \label{E}
  \end{equation}
\end{theorem}

We need $p>2$ to ensure (\ref{08}), a key estimate in verifying local linking. If $p=2$, instead of $(V_0)$ it is more natural to impose conditions on the  quadratic form
\[
    \mathscr{Q} (u) = \frac{1}{2} \int (| \Delta u |^2 +|\nabla u|^2+ V (x) u^2)
\]
and apply the classical linking theorems (see \citet[Thm 5.2]{MR845785}, \citet[Thm 6.10]{MR1400007}) to get solutions when $\mathscr{Q}$ is indefinite. In particular, we can study the case where $\Phi$ defined in \eqref{B} satisfies at most the local Palais-Smale condition. This is a broad area that is not pursued in the present paper.\medskip

The solutions of (\ref{e1}) will be obtained as critical points of the $C^1$-functional
$\Phi : X \rightarrow \mathbb{R}$,
\begin{equation}
  \Phi (u) = \frac{1}{2} \int (| \Delta u |^2 + V (x) u^2) + \frac{1}{p} \int
  | \nabla u |^p - \int F (x, u) \text{.} \label{B}
\end{equation}
If $\inf_{\mathbb{R}^N} V > 0$, then the quadratic form $Q$ is positive
definite and the zero function $u = 0$ is a local minimizer of $\Phi$, the
mountain pass theorem {\cite{MR0370183,MR845785}} can be applied to get a
nontrivial solution of (\ref{e1}). If the $p$-Laplacian term $\Delta_p u$ is
removed from (\ref{e1}), then the term involving $| \nabla u |^p$ does not
show up in $\Phi (u)$. As a consequence $\Phi$ verifies the linking geometry
and the linking theorem {\citet[Thm 5.3]{MR845785}} can be applied. In our
case, due to the presence of the nonnegative term $\int | \nabla u |^p$, $\Phi$ is not nonpositive
on the negative space of $Q$ and the linking geometry is not satisfied anymore. As in the works
{\cite{MR3303004,MR3656292}} on nonlinear indefinite Schr{\"o}dinger--Poisson
systems, we will apply local linking {\cite{MR802575,MR1312028}} and Morse theory
{\cite{MR1196690}} to get the desired solutions of our problem (\ref{e1}).

Note that $(f_1)$ is a version of the classical Ambrosetti--Rabinowitz
condition introduced in {\cite{MR0370183}} which implies
\begin{equation}
  \lim_{| t | \rightarrow \infty} \frac{F (x, t)}{| t |^p} = + \infty
  \text{\qquad a.e.\ $x \in \mathbb{R}^N$.} \label{ep}
\end{equation}
A typical $f$ satisfying the above assumptions is $f (x, t) = a (x) | t |^{q -
2} t$ for some $q>p$ and continuous $a : \mathbb{R}^N \rightarrow (0, \infty)$ decaying
to zero.

As in {\cite{MR3656292}}, the condition $(f_2)$ is borrowed from
{\cite{MR2038142}} and will \emph{only} be used to derive
\begin{equation}
  \varlimsup_{n \rightarrow \infty} \int f (x, u_n) (u_n - u) \leq 0 \text{}
  \label{Wi}
\end{equation}
from $u_n \rightharpoonup u$ in $X$. Hence, in Theorem \ref{t1} we can replace
$(f_2)$ by other condition ensuring this compact property, such as
\begin{equation}
  | f (x, t) | \leq \alpha_1 (x) | t |^{q_1 - 1} + \alpha_2 (x) | t |^{q_2 -
  1} \label{F}
\end{equation}
for some $q_i \in [2, 2_{\ast})$ and nonnegative $\alpha_i \in L^{q_i^{\#}}
(\mathbb{R}^N) \cap L_{\mathrm{loc}}^{\infty} (\mathbb{R}^N)$, where $q_i^{\#}
\in (1, \infty)$ can be arbitrary for $N \leq 4$ and $q_i^{\#} =
\frac{2_{\ast}}{2_{\ast} - q_i}$ for $N \geq 5$.

\begin{theorem}
  \label{t2}Suppose $(V_0)$, $(f_{\ast})$, $(f_1)$ and \eqref{F} are
  satisfied with $p\in(2,2^*]$.
  \begin{enumerate}
    \item[(a)] If $(f_0)$ is satisfied, then \eqref{e1} has a nontrivial solution.
    
    \item[(b)] If $f (x, \cdot)$ is odd, then \eqref{e1} has a sequence of
    solutions $\{ u_n \}$ satisfying \eqref{E}.
  \end{enumerate}
\end{theorem}

The Sobolev embedding $X \hookrightarrow L^2 (\mathbb{R}^N)$ may not be
compact. Therefore we need $(f_2)$ or (\ref{F}) to guarantee the global
Palais--Smale $(\tmop{PS})$ condition for $\Phi$, which is
{\emph{indispensable}} for applying local linking and Morse theory. If the
potential $V$ gives rise to compactness of the embedding, then (\ref{Wi}) is
guaranteed and both conditions $(f_2)$ and (\ref{F}) can be dropped.

\begin{theorem}
  \label{t3}Suppose $(V_0)$, $(f_{\ast})$ and $(f_1)$ are satisfied with $p\in(2,2^*]$, $V$ is
  coercive.
  \begin{enumerate}
    \item[(a)] If $(f_0)$ is satisfied, then \eqref{e1} has a nontrivial solution.
    
    \item[(b)] If $f (x, \cdot)$ is odd, then \eqref{e1} has a sequence of
    solutions $\{ u_n \}$ satisfying \eqref{E}.
  \end{enumerate}
\end{theorem}


The above theorems will be proved by looking for critical points of the variational functional $\Phi$, which satisfies the global $(\tmop{PS})$ condition under the assumptions of the theorems. In the last section of this paper, we will also study the case that the variational functional only satisfies the local $(\tmop{PS})$ condition. For this case, we can only handle the definite case that $\inf_{\mathbb{R}^N}V>0$. 
Consider
\begin{equation}
\left\{ \begin{array}{l}
  \Delta^2 u - \Delta_p u + V (x) u = \lambda g (x) | u |^{q - 2} u + | u |^{m
  - 2} u\text{,}\\
    u \in H^2 (\mathbb{R}^N) \text{,}
  \end{array} \right. \label{he2}
\end{equation}
where $2\le p<q<m\le 2_*$. 
Note that if $m = 2_{\ast}$ the problem (\ref{he2}) is of critical growth.

\begin{theorem}
  \label{ht5}Let $V \in L^{\infty} (\mathbb{R}^N)$, $\inf_{\mathbb{R}^N} V >
  0$; $g \in L^{2_{\ast} / (2_{\ast} - q)} (\mathbb{R}^N)$, $g > 0$ on some
  nonempty open set $\Omega$. If $2 < p \le q < m \leq 2_{\ast}$ or $2 = p < q < m\le 2_\ast$, then given
  $k \in \mathbb{N}$, there is $\lambda_0 > 0$ such that if $\lambda>\lambda_0$ then \eqref{he2} has $k$
  pairs of solutions.
\end{theorem}

The paper is organized as follows. In {\textsection}2 we prove some
preliminary lemmas that will be needed for proving our results. In
{\textsection}3 we prove Theorem \ref{t1}. The proofs of Theorems \ref{t2} and
\ref{t3} are similar to that of Theorem \ref{t1} and will be sketched in {\textsection}4. Finally, we prove Theorem \ref{ht5} in \S5.

\section{Some preliminary lemmas}

In this section, following {\cite{MR3656292}} we prove some lemmas, which will
be needed for the proofs of Theorems \ref{t1}, \ref{t2} and \ref{t3}. Among these lemmas, Lemma \ref{9} will also be needed for proving Theorem \ref{ht5}. By
assumption $(V_0)$, there is an equivalent norm $\| \cdot \|$ on $X$ such
that
\[ Q (u) = \frac{1}{2} (\| u^+ \|^2 - \| u^- \|^2)\text{,} \]
as a consequence we can rewrite $\Phi$ as
\begin{equation}
  \Phi (u) = \frac{1}{2} (\| u^+ \|^2 - \| u^- \|^2) + \frac{1}{p} \int |
  \nabla u |^p - \int F (x, u) \text{,} \label{pm}
\end{equation}
here and below for $u \in X$ we write $u = u^+ + u^-$ with $u^{\pm} \in
X^{\pm}$ being $X^{\pm}$ the positive/negative spaces of $Q$. In other words, $u^\pm$ are the orthogonal projections of $u$ in the positive/negative spaces of $Q$.

\begin{lemma}
  \label{8}Under the assumptions $(V_0)$, $(f_{\ast})$ and $(f_0)$ with $p\in(2,2^*]$, the functional
  $\Phi : X \rightarrow \mathbb{R}$ has a {\emph{local linking}} at $u = 0$.
  That is
  \begin{equation}
    \begin{array}{lll}
      \Phi (u) \leq 0 &  & \text{for $u \in X^- \cap B_r$,}\\
      \Phi (u) > 0 &  & \text{for $u \in (X^+ \backslash \{ 0 \}) \cap B_r$}
    \end{array} \label{e5}
  \end{equation}
  for some $r > 0$, where $B_r = \{ u \in X \mid \| u \| < r \}$.
\end{lemma}

\begin{proof}
  Given $\varepsilon > 0$, by $(f_{\ast})$ and $(f_0)$ there is
  $C_{\varepsilon} > 0$ such that
  \[ | F (x, t) | \leq \varepsilon | t |^2 + C_{\varepsilon} | t |^q
     \text{\qquad  for $(x, t) \in \mathbb{R}^N \times \mathbb{R}$.} \]
  It follows that
  \[ \int F (x, u) = o (\| u \|^2) \text{\qquad as $\| u \| \rightarrow 0$.}
  \]
  For $u \in X$ we have $\nabla u \in H^1$. By the continuous embedding $H^1
  \hookrightarrow L^p$ we get
  \begin{equation}
    \int | \nabla u |^p = | \nabla u |_p^p \leq C_1 \| \nabla u \|_{H^1}^p
    \leq C_2 \| u \|^p \text{.} \label{Po}
  \end{equation}
  Consequently
  \begin{align} \int | \nabla u |^p = o (\| u \|^2) \text{\qquad as $\| u \| \rightarrow
     0$.} \label{08}
\end{align}
  In summary
  \begin{align}
    \Phi (u) & =  \frac{1}{2} (\| u^+ \|^2 - \| u^- \|^2) + \frac{1}{p} \int
    | \nabla u |^p - \int F (x, u) \nonumber\\
    & =  \frac{1}{2} (\| u^+ \|^2 - \| u^- \|^2) + o (\| u \|^2) \label{12} 
  \end{align}
  as $\| u \| \rightarrow 0$. Using (\ref{12}), we deduce (\ref{e5})
  immediately.
\end{proof}

\begin{lemma}
  \label{l2}Under assumptions $(V_0)$, $(f_{\ast})$ and $(f_1)$, any
  $(\tmop{PS})$ sequences of $\Phi : X \rightarrow \mathbb{R}$ are bounded.
\end{lemma}

\begin{proof}
  Let $\{ u_n \} \subset X$ be a $(\tmop{PS})$ sequence of $\Phi$, that is
  $\Phi (u_n) \rightarrow c$ for some $c \in \mathbb{R}$ and $\nabla\Phi (u_n)
  \rightarrow 0$. If $\{ u_n \}$ is unbounded we may assume $\| u_n \|
  \rightarrow \infty$. Set $v_n = \| u_n \|^{- 1} u_n$. Then
  \[ v_n = v_n^+ + v_n^- \rightharpoonup v = v^+ + v^- \text{\quad in $X$,
     \qquad$v_n \rightarrow v$\quad a.e.\ on $\mathbb{R}^N$} \]
  being $v_n^{\pm}, v^{\pm} \in X^{\pm}$. We will derive contradiction for
  both cases $v = 0$ and $v \neq 0$, this will prove the lemma.
  
  If $v = 0$, then $v_n^- \rightarrow v^- = 0$ because $\dim X^- < \infty$.
  For $n \gg 1$ we have
  \[ \| v_n^+ \|^2 - \| v_n^- \|^2 \geq \frac{1}{2} \]
  because $\| v_n^+ \|^2 + \| v_n^- \|^2 = 1$. Using $(f_1)$ we have
  \begin{align*}
    c + 1 + \| u_n \| & \geq  \Phi (u_n) - \frac{1}{\mu} \langle \nabla\Phi (u_n),
    u_n \rangle\\
    & =  \left( \frac{1}{2} - \frac{1}{\mu} \right) (\| u_n^+ \|^2 - \|
    u_n^- \|^2) + \left( \frac{1}{p} - \frac{1}{\mu} \right) \int | \nabla u_n
    |^p\\
    &   \qquad - \int \left( F (x, u_n) - \frac{1}{\mu} f (x, u_n) u_n
    \right)\\
    & \geq  \left( \frac{1}{2} - \frac{1}{\mu} \right) \| u_n \|^2 (\| v_n^+
    \|^2 - \| v_n^- \|^2) \geq \left( \frac{1}{4} - \frac{1}{2 \mu} \right) \|
    u_n \|^2 \text{,}
  \end{align*}
  contradicting $\| u_n \| \rightarrow \infty$.
  
  If $v \neq 0$, then $\Theta = \{ v \neq 0 \}$ has positive Lebesgue measure.
  For $x \in \Theta$ we have $| u_n (x) | \rightarrow \infty$ and
  \[ \frac{F (x, u_n (x))}{| u_n (x) |^p} | v_n (x) |^p \rightarrow + \infty
  \]
  because of (\ref{ep}), which follows from $(f_1)$. Applying Fatou lemma we get
  \begin{equation}
    \frac{1}{\| u_n \|^p} \int F (x, u_n) \geq \int_{\Theta} \frac{F (x,
    u_n)}{| u_n |^p} | v_n |^p \rightarrow + \infty \text{.} \label{cr}
  \end{equation}
  However, for $n \gg 1$ using (\ref{Po}) we see that
  \begin{align*}
    \frac{1}{\| u_n \|^p} \int F (x, u_n) & =  \frac{1}{\| u_n \|^p} \left( Q
    (u_n) + \frac{1}{p} \int | \nabla u_n |^p - \Phi (u_n) \right)\\
    & \leq  1 + \frac{C_2}{p}
  \end{align*}
  is bounded above, contradicting (\ref{cr}).
\end{proof}

To study the convergence of $(\tmop{PS})$ sequences, we need to study the
$C^1$-functional $\mathcal{N}: X \rightarrow \mathbb{R}$,
\[ \mathcal{N} (u) = \frac{1}{p} \int | \nabla u |^p \text{.} \]
\begin{lemma}
  \label{9}The functional $\mathcal{N}$ is weakly lower semi-continuous, its
  gradient $\nabla \mathcal{N}: X \rightarrow X$ is weakly continuous. That
  is, $u_n \rightharpoonup u$ in $X$ implies
  \begin{align} \mathcal{N} (u) \leq\varliminf_{n \rightarrow \infty} \mathcal{N} (u_n)
     \text{, \qquad} \langle \nabla \mathcal{N}(u), \phi \rangle = \lim_{n
     \rightarrow \infty} \langle \nabla \mathcal{N}(u_n), \phi \rangle \label{ds7}
     \end{align}
  for all $\phi \in X$; where $\langle \cdot, \cdot \rangle$ is the inner
  product on $X$.
\end{lemma}

\begin{proof}
Let $\{u_{n_k}\}$ be a subsequence such that
\[
  \mathcal{N} (u_{n_k})\to \varliminf_{n \rightarrow \infty} \mathcal{N} (u_n)\text{.} 
 \]
Since $u_{n_k} \rightharpoonup u$ in $X$, up to
a subsequence we have $u_{n_k} \rightarrow u$ in
$H^1_{\mathrm{loc}}$, $\nabla u_{n_k} \rightarrow \nabla u$ in
$L^2_{\mathrm{loc}}$ and $\nabla u_{n_k} \rightarrow \nabla u$ a.e.\ on
$\mathbb{R}^N$. It follows from Fatou's Lemma that
\begin{align*}
  \mathcal{N} (u) & =  \frac{1}{p}  \int \varliminf_{k \rightarrow \infty} |
  \nabla u_{n_k} |^p\\
  & \leq  \frac{1}{p} \varliminf_{k \rightarrow \infty} \int | \nabla u_{n_k}
  |^p = \varliminf_{k \rightarrow \infty} \mathcal{N} (u_{n_k}) = \varliminf_{n \rightarrow \infty} \mathcal{N} (u_n) \text{.}
\end{align*}

Now we prove the second result in (\ref{ds7}). For $N = 2$, set $q = 2$; for $N \geq 3$ set
\[ q = \left\{ \begin{array}{lll}
     2 &  & \text{if $2 \leq p \leq \frac{2 N - 2}{N - 2}$,}\\[1em]
     2^{\ast} &  & \text{if $\frac{2 N - 2}{N - 2} < p \leq 2^{\ast}$.}
   \end{array} \right. \]
Then it is easy to check that $q'  (p - 1) \in [2, 2^{\ast}]$ being $q' = q /
(q - 1)$.

Given $\phi \in X$, we consider an arbitrary subsequence $\{u_{n_k} \}$. Then
$\{\nabla u_{n_k} \}$ is bounded in $H^1$, hence in $L^{q'  (p - 1)}$.
Therefore
\[ \sup_k \left| \int \big| | \nabla u_{n_k} |^{p - 2} \nabla u_{n_k} \big|^{q'}
   \right| = \sup_k  \int | \nabla u_{n_k} |^{q'  (p - 1)} < \infty \text{,}
\]
so $\{| \nabla u_{n_k} |^{p - 2} \nabla u_{n_k} \}$ is bounded in $L^{q'}$. Up
to a subsequence $\nabla u_{n_k} \rightarrow \nabla u$ a.e. on $\mathbb{R}^N$,
by a well-known remark of {\citet[pp. 487]{MR699419}} (see
{\citet[Lem 1.4.8]{MR1276944}} for a proof), we deduce
\[ | \nabla u_{n_k} |^{p - 2} \nabla u_{n_k} \rightharpoonup | \nabla u|^{p -
   2} \nabla u \text{\qquad in $L^{q'}$.} \]
Therefore, since $\nabla \phi \in L^q$, we conclude
\begin{align*}
  \lim_{k \rightarrow \infty} \langle \nabla \mathcal{N}(u_{n_k}), \phi
  \rangle & =  \lim_{n \rightarrow \infty}  \int | \nabla u_{n_k} |^{p - 2}
  \nabla u_{n_k} \cdot \nabla \phi\\
  & =  \int | \nabla u|^{p - 2} \nabla u \cdot \nabla \phi = \langle \nabla
  \mathcal{N}(u), \phi \rangle \text{.}
\end{align*}
We have shown that any subsequence of $\{\langle
\nabla \mathcal{N} (u_n), \phi \rangle\}$ has a subsequence converges to $\langle \nabla \mathcal{N}
(u), \phi \rangle$. Thus  as desired, $\langle
\nabla \mathcal{N} (u_n), \phi \rangle \rightarrow \langle \nabla \mathcal{N}
(u), \phi \rangle$.
\end{proof}

\begin{lemma}
  \label{l4}Under the assumptions of Theorem \ref{t1}, $\Phi : X \rightarrow
  \mathbb{R}$ satisfies the $(\tmop{PS})$ condition.
\end{lemma}

\begin{proof}
  Let $\{ u_n \} \subset X$ be a $(\tmop{PS})$ sequence of $\Phi$. By Lemma
  \ref{l2}, $\{ u_n \}$ is bounded. Up to a subsequence $u_n \rightharpoonup
  u$ in $X$ for some $u \in X$, and $\| u_n^- \| \rightarrow \| u^- \|$
  because $\dim X^- < \infty$.
  
  It suffices to show that $\| u_n \| \rightarrow \| u \|$ (then $u_n\to u$ follows because $X$ is a Hilbert space). Noting
  \[ \langle \nabla \mathcal{N} (u_n), u_n - u \rangle = p\mathcal{N} (u_n) -
     \langle \nabla \mathcal{N} (u_n), u \rangle \text{,} \]
  by direct computation we deduce
  \begin{align}
    o (1) & =  \langle \nabla\Phi (u_n), u_n - u \rangle \nonumber\\
    & =  (\| u_n^+ \|^2 - \| u_n^- \|^2) - (\langle u_n^+, u^+ \rangle -
    \langle u_n^-, u^- \rangle) \nonumber\\
    &   \qquad\qquad + p\mathcal{N} (u_n) - \langle \nabla \mathcal{N} (u_n), u
    \rangle - \int f (x, u_n) (u_n - u) \nonumber\\
    & =  (\| u_n^+ \|^2 - \| u_n^- \|^2) - (\| u^+ \|^2 - \| u^- \|^2) + o
    (1) \nonumber\\
    &   \qquad\qquad + p\mathcal{N} (u_n) - \langle \nabla \mathcal{N} (u_n), u
    \rangle - \int f (x, u_n) (u_n - u) \text{.} \label{7} 
  \end{align}
  Similar to {\citet[p. 29]{MR2038142}}, using $(f_2)$ we can prove
  \begin{equation}
    \varlimsup_{n \rightarrow \infty} \int f (x, u_n) (u_n - u) \leq 0 \text{.}
    \label{W}
  \end{equation}
  Thus, using Lemma \ref{9} and $\| u_n^- \| \rightarrow \| u^- \|$, we
  deduce from (\ref{7}) and (\ref{W}) that
  \begin{align*}
    \| u^+ \|^2 & \leq  \varliminf_{n \rightarrow \infty} \| u_n^+ \|^2 \leq
    \varlimsup_{n \rightarrow \infty} \| u_n^+ \|^2\\
    & =  \| u^+ \|^2 + \varlimsup_{n \rightarrow \infty} \left( \langle \nabla
    \mathcal{N} (u_n), u \rangle - p\mathcal{N} (u_n) + \int f (x, u_n) (u_n -
    u) \right)\\
    & \leq  \| u^+ \|^2+ \langle \nabla
    \mathcal{N} (u), u \rangle - p\mathcal{N} (u)=\| u^+ \|^2 \text{.}
  \end{align*}
  It follows that $\| u_n^+ \| \rightarrow \| u^+ \|$. Noting $\| u_n^- \|
  \rightarrow \| u^- \|$ we get $\| u_n \| \rightarrow \| u \|$.
\end{proof}

\begin{lemma}
  \label{4}Under the assumptions $(V_0)$, $(f_{\ast})$ and $(f_1)$, there is $A >
  0$ such that if $\Phi (u) \leq - A$ then
  \begin{align} \left. \frac{\mathrm{d}}{\mathrm{d} t} \right|_{t = 1} \Phi (t u) < 0 \text{.} \label{Er}
  \end{align}
\end{lemma}

\begin{proof}
  If the claim is not true, there exists a sequence $\{ u_n \} \subset X$ such
  that $\Phi (u_n) \leq - n$ but
  \begin{equation}
    \langle \nabla\Phi (u_n), u_n \rangle = \left. \frac{\mathrm{d}}{\mathrm{d} t}
    \right|_{t = 1} \Phi (t u_n) \geq 0 \text{.} \label{pp}
  \end{equation}
  It follows from $(f_1)$ that
  \begin{align}
    \left( \frac{\mu}{2} - 1 \right) (\| u_n^+ \|^2 - \| u_n^- \|^2) & \leq 
    \frac{\mu - 2}{2} (\| u_n^+ \|^2 - \| u_n^- \|^2) + \frac{\mu - p}{p} \int
    | \nabla u_n |^p \nonumber\\
    &   \hspace{7em} + \int (f (x, u_n) u_n - \mu F (x, u_n)) \nonumber\\
    & =  \mu \Phi (u_n) - \langle \nabla\Phi (u_n), u_n \rangle \leq - \mu n
    \text{.} \label{5} 
  \end{align}
  Let $v_n = \| u_n \|^{- 1} u_n$. Then since $\dim X^- < \infty$ we have
  \[ v_n \rightharpoonup v \text{\quad in $X$, \qquad$v_n^- \rightarrow
     v^-$\quad in $X^-$.} \]
  If $v^- = 0$, From $\| v_n^+ \|^2 + \| v_n^- \|^2 = 1$ we deduce $\| v_n^+
  \| \rightarrow 1$. Hence
  \[ \| u_n^+ \|^2 = \| u_n \|^2 \| v_n^+ \|^2 \geq \| u_n \|^2 \| v_n^- \|^2
     = \| u_n^- \|^2 \]
  for $n\gg1$, contradicting (\ref{5}).
  
  Thus, we must have $v^- \neq 0$. So $v \neq 0$ as well. Similar to
  (\ref{cr}) we get
  \[ \frac{1}{\| u_n \|^p} \int f (x, u_n) u_n \geq \frac{\mu}{\| u_n \|^p}
     \int F (x, u_n) \rightarrow + \infty \text{.} \]
  From this, (\ref{pp}) and (\ref{Po}), we get a contradiction
  \begin{align*}
    0 & \leq  \frac{\langle \nabla\Phi (u_n), u_n \rangle}{\| u_n \|^p}\\
    & =  \frac{1}{\| u_n \|^p} \left( 2Q (u_n) + \int | \nabla u_n |^p - \int
    f (x, u_n) u_n \right) \rightarrow - \infty \text{.}
  \end{align*}
\end{proof}

\begin{lemma}
  \label{10}Assuming $(f_{\ast})$ and $(f_1)$, $\Phi : X \rightarrow
  \mathbb{R}$ is anti-coercive on any finite dimensional subspace of $X$.
\end{lemma}

\begin{proof}
  Let $\{ u_n \}$ be a sequence in a finite dimensional subspace $Y$
  satisfying $\| u_n \| \rightarrow \infty$. Set $v_n = \| u_n \|^{- 1} u_n$.
  Since $\dim Y < \infty$, we may assume
  \[ \| v_n - v \| \rightarrow 0 \text{, \qquad$v_n \rightarrow v$\quad a.e.\
     on $\mathbb{R}^N$} \]
  for some $v \in Y \cap \partial B_1$, in particular $v \neq 0$. Similar to
  (\ref{cr}) we have
  \begin{equation}
       \frac{1}{\| u_n \|^p} \int F (x, u_n) \rightarrow + \infty \text{.} \label{Q}
  \end{equation}
  The desired result follows from the following consequence of (\ref{Q})
  \begin{align*}
    \Phi (u_n) & =  \| u_n \|^p \left( \frac{Q (u_n)}{\| u_n \|^p} +
    \frac{1}{\| u_n \|^p} \int | \nabla u_n |^p - \frac{1}{\| u_n \|^p} \int F
    (x, u_n) \right) \rightarrow - \infty \text{.}
  \end{align*}
  
\end{proof}

\section{Proof of Theorem \ref{t1}}

As mentioned before, solutions of (\ref{e1}) will be found as critical points
of $\Phi : X \rightarrow \mathbb{R}$ given in (\ref{B}) or equivalently,
(\ref{pm}). To this end we need some concepts and results from infinite
dimensional Morse theory, see {\cite{MR1196690}} or {\citet[Chp 8]{MR982267}} for an excellent account of the theory.

Let $\varphi : X \rightarrow \mathbb{R}$ be a $C^1$-functional on a Banach
space $X$. If $u \in X$ is an isolated critical point of $\varphi$ with
$\varphi (u) = c$, we call
\[ C_k (\varphi, u) = H_k (\varphi^c, \varphi^c \backslash \{ u \}) \text{,
   \qquad$k = 0, 1, 2, \ldots$} \]
the $k$-th critical group of $\varphi$ at $u$, where $\varphi^c = \varphi^{-
1} (- \infty, c]$, $H_k$ stands for the $k$-th singular homology group of the
topological pair with coefficients in $\mathbb{Q}$.

If $\varphi : X \rightarrow \mathbb{R}$ satisfies the $(\tmop{PS})$ condition
and $- \infty < \alpha < \inf_{\mathcal{K}} \varphi$, following
{\cite{MR1420790}}, we call
\[ C_k (\varphi, \infty) = H_k (X, \varphi^{\alpha}) \]
the $k$-th critical group of $\varphi$ at infinity. It should be pointed out
that by deformation lemma (a con\-s\-e\-q\-uence of the $(\tmop{PS})$ condition) the right
hand side does not depend on the choice of $\alpha$.

It is well known that $C_{\ast} (\varphi, u)$ characterize the local property
of $\varphi$ near $u$, while $C_{\ast} (\varphi, \infty)$ characterize the
global property of $\varphi$. When they are distinct, there must be a critical
point different from $u$.

\begin{proposition}[{\citet[Prop 3.6]{MR1420790}}]
  \label{p}If $\varphi \in C^1 (X, \mathbb{R})$ satisfies the $(\tmop{PS})$
  condition, $C_{\ell} (\varphi, 0) \neq C_{\ell} (\varphi, \infty)$ for some
  $\ell \in \mathbb{N}$, then $\varphi$ has a nonzero critical point.
\end{proposition}

To get multiple solutions of (\ref{e1}) when $f (x, \cdot)$ is odd, we need
the symmetric mountain pass theorem of 
{\cite{MR0370183}}. For $r>0$, we write $B_r$ for the $r$-ball at the origin.

\begin{proposition}[{\citet[Thm 9.12]{MR845785}}]
  \label{mp}Let $X$ be an infinite dimensional Banach space, $\varphi \in C^1
  (X, \mathbb{R})$ be even, satisfies $(\tmop{PS})$ condition and $\varphi (0)
  = 0$. If $\varphi$ satisfies
  \begin{enumerate}
    \item[$(I_1)$] for any finite dimensional subspace $W \subset X$, there is
    $R (W) > 0$ such that $\varphi \leq 0$ on $W \backslash B_{R (W)}$,
    
    \item[$(I_2)$] there are a finite codimensional subspace $Z$ and constants
    $\rho, \alpha > 0$ such that $ \varphi |_{\partial B_{\rho} \cap
    Z} \geq \alpha$,
  \end{enumerate}
  then $\varphi$ has a sequence of critical values $c_j \rightarrow + \infty$.
\end{proposition}

Now, we are ready to prove Theorem \ref{t1}.

\begin{proof}[Proof of Theorem \ref{t1}]To find a nontrivial solution, we may assume that $u=0$ is the \emph{only} critical point of $\Phi$, since otherwise any other critical points would already be nontrivial solution of \eqref{e1}.

By Lemma \ref{l4}, $\Phi$ satisfies the $(\tmop{PS})$ condition. Lemma
  \ref{8} says that $\Phi$ has a local linking at $u = 0$, which by {\citet[Thm
  2.1]{MR1110119}} implies
  \begin{align}
  C_{\ell} (\Phi, 0) \neq 0 \text{,} \label{Cc}
  \end{align}
  where $\ell = \dim X^-$. For $u\in\partial B_1$, by Lemma \ref{10} we have
  \[ \Phi (t u) \rightarrow - \infty \text{\qquad as $t \rightarrow +
     \infty$.} \]
  Thus as in the proof of {\citet[Lem 3.4]{MR3656292}} (which follows idea
  tracing back to {\citet{MR1094651}}), by Lemma \ref{4}, for some
  \[ A > \sup_{\| u \| \leq 2} | \Phi (u) | \]
satisfying (\ref{Er}), there is a unique $t_u > 0$ such that $\Phi (t_u u) = - A$. By (\ref{Er}) and the
  implicit function theorem $T : u \mapsto t_u$ is a continuous function on $\partial B_1$. Using the
  function $T$ we define a retraction $\eta : X \backslash
  B_1 \rightarrow \Phi^{- A}$ via
  \[ \eta (u) = \Bigg\{ \begin{array}{lll}
       u &  & \text{if $\Phi (u) \leq - A$,}\\
       T \bigg( \dfrac{u}{\| u \|} \bigg) \dfrac{u}{\| u \|} &  & \text{if
       $\Phi (u) > - A$.}
     \end{array} \]
Clearly $\Phi_{-A}\subset X\backslash B_1$.  Let $i:\Phi_{-A}\to X\backslash B_1$ be the inclusion.   Define $H : [0, 1] \times (X \backslash B_1) \rightarrow X \backslash B_1$ via
\[ H (t, u) = (1 - t) u + t \cdot (i\circ \eta) (u) \text{.} \]
It is clear that $\eta \circ i =\mathrm{id}_{\Phi_{- A}}$, and $i \circ \eta \simeq \mathrm{id}_{X \backslash B_1}$ via the homotopy $H$. It
follows that $\Phi_{- A} \simeq X \backslash B_1$, which implies
  \begin{align}
   C_k (\Phi, \infty) = H_k (X, \Phi_{- A}) \cong H_k (X, X \backslash B_1)
     = 0 \text{, \qquad$k = 0, 1, \ldots$} \label{cC}
     \end{align}
  because $\dim X = \infty$.
  
 From (\ref{Cc}) and (\ref{cC}) we have $C_{\ell} (\Phi, 0) \neq C_{\ell} (\Phi, \infty)$. By Proposition
  \ref{p} our $\Phi$ has a nonzero critical point and the first part of
  Theorem \ref{t1} is proved.
  
  If $f (x, \cdot)$ is odd then $\Phi : X \rightarrow \mathbb{R}$ is even.
  To get a sequence of critical values of $\Phi$ and derive the second part of
  Theorem \ref{t1} via Proposition \ref{mp}, it suffices to verify condition
  $(I_2)$ because $(I_1)$ follows directly from Lemma \ref{10}.
  
  We take $Z = X^+$, the positive space of $Q$. There is $\kappa > 0$ such
  that
  \begin{equation}
    \frac{1}{2} \int (| \Delta u |^2 + V (x) u^2) \geq \kappa \| u \|^2
    \text{\qquad for $u \in Z$.} \label{Z}
  \end{equation}
  For $m \in [2, 2_{\ast}]$, let $S_m$ be the norm of the continuous embedding
  $X \hookrightarrow L^m (\mathbb{R}^N)$. Due to $(f_{\ast})$ and $(f_0)$,
  there is $C > 0$ such that
  \[ | F (x, t) | \leq \frac{\kappa}{2 S_2^2} | t |^2 + C | t |^q \text{.} \]
  Choose $\rho > 0$ small enough such that
  \[ r: =\frac{\kappa}{2} - C S_q^q \rho^{q - 2} > 0 \text{.} \]
  Then, for $u \in \partial B_{\rho} \cap Z$
  \begin{align*}
    \Phi (u) & =  \frac{1}{2} \int (| \Delta u |^2 + V (x) u^2) + \frac{1}{p}
    \int | \nabla u |^p - \int F (x, u)\\
    & \geq  \kappa \| u \|^2 - \int F (x, u)\\
    & \geq  \kappa \| u \|^2 - \frac{\kappa}{2 S_2^2} | u |_2^2 - C | u
    |_q^q\\
    & \geq  \frac{\kappa}{2} \| u \|^2 - C S_q^q \| u \|^q = \rho^2 r
    \text{.}
  \end{align*}
  This verifies $(I_2)$ with $\alpha = \rho^2 r$.
\end{proof}

\section{Proofs of Theorems \ref{t2} and \ref{t3}}

Checking the proof of Theorem \ref{t1}, we see that $(f_2)$ is only used to
derive (\ref{W}) from $u_n \rightharpoonup u$ in $X$, and for getting multiple
solutions $(f_0)$ is only needed for verifying condition $(I_2)$ of
Proposition \ref{mp}. Therefore, to prove Theorems \ref{t2} and \ref{t3}, it
suffices to derive (\ref{W}) assuming (\ref{F}) or the coerciveness of $V$,
and verify $(I_2)$ without assuming $(f_0)$.

\begin{proposition}
  \label{p4}Under the assumptions $(V_0)$ and $(f_{\ast})$, if $u_n \rightharpoonup u$ in $X$, then
  \begin{equation}
    \int f (x, u_n) (u_n - u) \rightarrow 0 \text{,} \label{P}
  \end{equation}
  provided either \eqref{F} holds or $V$ is
  coercive.
\end{proposition}

\begin{proof}
  If $V$ is coercive, for $m \in [2, 2_{\ast})$ the embedding $X
  \hookrightarrow L^m (\mathbb{R}^N)$ is compact. Thus $u_n \rightarrow u$ in
  $L^m (\mathbb{R}^N)$. By $(f_{\ast})$ and H{\"o}lder inequality we get
  \begin{align*}
    \left| \int f (x, u_n) (u_n - u) \right| & \leq  C \int (| u_n | + | u_n
    |^{q - 1}) | u_n - u |\\
    & \leq  C (| u_n |_2 | u_n - u |_2 + | u_n |_q^{q - 1} | u_n - u |_q)
    \rightarrow 0 \text{.}
  \end{align*}
  
Now assume that (\ref{F}) holds. Let
\begin{align} r_i = \left\{ \begin{array}{ll}
     2_{\ast} & \text{if $N \geq 5$,}\\
     q_i (q_i^{\#})' & \text{if $N \leq 4$.}
   \end{array} \right. \label{r}
   \end{align}
Then $r_i\in[2,\infty)$ and we have a continuous embedding $X \hookrightarrow L^{r_i}$. Thus $\sup_n
|u_n |_{r_i} < \infty$.

Given $\varepsilon > 0$, since $\alpha_i \in L^{q_i^{\#}}$ and
\[ \frac{1}{q_i^{\#}} + \frac{q_i - 1}{r_i} + \frac{1}{r_i} = 1 \text{,} \]
applying H{\"o}lder inequality we have
\[ \int_{|x| \geq R} \alpha_i (x) |u_n |^{q_i - 1}  |u_n - u| \leq \left(
   \int_{|x| \geq R} | \alpha_i (x) |^{q_i^{\#}} \right)^{1 / q_i^{\#}} |u_n
   |_{r_i}^{q_i - 1}  |u_n - u|_{r_i} \leq \varepsilon \]
for some $R > 1$. Noting $\alpha_i \in L_{\mathrm{loc}}^{\infty}$, it follows
from (\ref{F}) that
\begin{align}
  \left| \int f (x, u_n) (u_n - u) \right| & \leq \sum_{i = 1}^2 \left(
  \int_{|x| \geq R} + \int_{|x| < R} \right) \alpha_i (x) |u_n |^{q_i - 1} 
  |u_n - u| \nonumber\\
  & \leq 2 \varepsilon + \sum_{i = 1}^2 | \alpha_i |_{\infty, B_R}  \int_{|x|
  < R} |u_n |^{q_i - 1}  |u_n - u| \nonumber\\
  & \leq 2 \varepsilon + \sum_{i = 1}^2 | \alpha_i |_{\infty, B_R} |u_n
  |_{q_i, B_R}^{q_i - 1}  |u_n - u|_{q_i, B_R} \text{,}  \label{16}
\end{align}
where for $m \in [2, \infty]$ and $D \subset \mathbb{R}^N$, we write $|
\cdot |_{m, D}$ for the $L^m (D)$ norm. Since $q_i \in [2, 2_{\ast})$, by
the compact embedding $X \hookrightarrow L_{\mathrm{loc}}^{q_i}$ we have $|u_n
- u|_{q_i, B_R} \rightarrow 0$. Thus
\[ \limsup_{n \rightarrow \infty} \left| \int f (x, u_n) (u_n - u) \right|
   \leq 2 \varepsilon \]
and (\ref{P}) follows by letting $\varepsilon \rightarrow 0$.
\end{proof}

\begin{remark}
  \label{r14}Using similar argument, we can show that $\Psi : X \rightarrow
  \mathbb{R}$ defined via
  \begin{equation}
    \Psi (u) = \int F (x, u) \label{17}
  \end{equation}
  is weakly continuous at $u = 0$. Let's assume (\ref{F}), which implies
  \[ | F (x, t) | \leq \frac{1}{q_1} \alpha_1 (x) | t |^{q_1} + \frac{1}{q_2}
     \alpha_2 (x) | t |^{q_2} \text{.} \]
  Given $\varepsilon > 0$, take $R > 0$ such that for the $r_i$ given in (\ref{r})
  \[ \frac{1}{q_i} \int_{| x | \geq R} \alpha_i (x) | u |^{q_i} \leq
     \frac{1}{q_i} \left( \int_{| x | \geq R} | \alpha_i (x) |^{q_i^{\#}}
     \right)^{1 / q_i^{\#}} | u_n |_{r_i}^{q_i} < \varepsilon \text{.} \]
  Similar to (\ref{16}) we deduce
  \begin{align*}
    | \Psi (u_n) - \Psi (0) | & =  \left| \int F (x, u_n) \right|\\
    & \leq  \sum_{i = 1}^2 \frac{1}{q_i} \left( \int_{| x | \geq R} +
    \int_{| x | < R} \right) \alpha_i (x) | u_n |^{q_i}\\
    & \leq  2 \varepsilon + \sum_{i = 1}^2 \frac{1}{q_i} | \alpha_i
    |_{\infty, B_R} | u_n |_{q_i, B_R}^{q_i}
  \end{align*}
  and the result follows because $| u_n |_{q_i, B_R} \rightarrow 0$. Actually it is possible to show that $\Psi$ is weakly continuous at all $u\in X$, but we don't need this result.
\end{remark}

Clearly Proposition \ref{p4} is stronger than (\ref{W}). Thus, as commented at
the beginning of this section, Part (a) of Theorems \ref{t2} and \ref{t3} has
been proved. To prove Part (b) of the theorems, we verify condition $(I_2)$ of
Proposition \ref{mp} below. Note that unlike Theorem \ref{t1}, condition $(f_0)$ is not assumed here.

Let $\{\phi_i \}_{i = 1}^{\ell}$ be an orthogonal basis of $X^-$. We extend
$\{\phi_i \}_{i = 1}^{\ell}$ to an orthogonal basis $\{\phi_i \}_{i =
1}^{\infty}$ of $X$. For $n \in \mathbb{N}$ set
\[ Y_n = \mathrm{span} \{\phi_1, \ldots, \phi_n \}  \text{, \qquad} Z_n =
   \overline{\mathrm{span}} \{\phi_i \mid i \geq n + 1\} \text{.} \]
Then $X = Y_n \oplus Z_n$, so all $Z_n$ are finite codimensional. Note that
$Z_{\ell} = X^+$.

By Remark \ref{r14}, we know that $\Psi : X \rightarrow \mathbb{R}$ defined in
(\ref{17}) is weakly continuous at $u = 0$. Applying {\cite[Lem
3.3]{MR2092084}} (checking its proof, in which $\{f_m\}$ is basis of $X^*$ satisfying $\langle f_m,\phi_n\rangle=\delta_{mn}$, weak continuity at $u = 0$ is
{\emph{sufficient}} for the conclusion), we deduce
\[ \beta_n = \sup_{Z_n \cap \partial B_1} | \Psi | \rightarrow 0 \text{.} \]
Take $n > \ell$ such that $\kappa - \beta_n > 0$. Then $Z=Z_n$ is finite codimensional.

For $u \in Z_n \cap \partial B_1$, since $Z_n \subset Z_{\ell}
= X^+$, using (\ref{Z}) we have
\begin{align*}
  \Phi (u) & \geq  \frac{1}{2} \int (| \Delta u |^2 + V (x) u^2) - \int F (x,
  u)\\
  & \geq  \kappa - \Psi (u) \geq \kappa - \beta_n \text{.}
\end{align*}
We deduce
\[ \inf_{Z_n \cap \partial B_1} \Phi \geq \kappa - \beta_n > 0 \text{} \]
and $(I_2)$ of Proposition \ref{mp} follows. This complete the proofs of
Theorems \ref{t2} and \ref{t3}.

\section{Proof of Theorem \ref{ht5}}

We write
\begin{equation}
  S = \inf_{u \in X \backslash \{ 0 \}} \frac{\| u \|^2}{| u |_m^2} \text{,
  \qquad} c^{\ast} = \left( \frac{1}{2} - \frac{1}{m} \right) S^{m / (m - 2)}
  \text{.} \label{hst}
\end{equation}
Consider the $C^1$-functional $\Phi : X \rightarrow \mathbb{R}$
\begin{align*}
  \Phi (u) & =  \frac{1}{2} \int (| \Delta u |^2 + V u^2) + \frac{1}{p} \int
  | \nabla u |^p - \frac{\lambda}{q} \int g | u |^q - \frac{1}{m} \int | u
  |^m\\
  & =  \frac{1}{2} \| u \|^2 + \frac{1}{p} \int | \nabla u |^p -
  \frac{\lambda}{q} \int g | u |^q - \frac{1}{m} \int | u |^m \text{.}
\end{align*}
It is clear that critical points of $\Phi$ are solutions of (\ref{he2}).

\begin{lemma}
  \label{h3l}Under the assumptions of Theorem \ref{ht5}, $\Phi$ satisfies $(\tmop{PS})_c$ for $c \in (0, c^{\ast})$.
\end{lemma}

\begin{proof}
  Let $\{ u_n \} \subset X$ be a $(\tmop{PS})_c$ sequence, that is $\Phi (u_n)
  \rightarrow c$, $\nabla \Phi (u_n) \rightarrow 0$. Then for $n\gg1$
  \begin{align*}
    c + 1 & \geq  \Phi (u_n) - \frac{1}{q} \langle \nabla \Phi (u_n), u_n
    \rangle\\
    & =  \left( \frac{1}{2} - \frac{1}{q} \right) \| u_n \|^2 + \left( \frac{1}{p} - \frac{1}{q} \right)\int|\nabla u_n|^p + \left(
    \frac{1}{q} - \frac{1}{m} \right) \int | u_n |^m\\
    & \geq  \left( \frac{1}{2} - \frac{1}{q} \right) \| u_n \|^2 \text{,}
  \end{align*}
  we see that $\{ u_n \}$ is bounded. Thus, up to a subsequence
  \[ u_n \rightharpoonup u \text{\quad in $X$, \qquad$u_n \rightarrow u$\quad
     a.e.\ on $\mathbb{R}^N$.} \]
  Clearly $\{ | u_n |^{m - 2} u_n \}$ is bounded in $L^{m / (m - 1)}$, $| u_n
  |^{m - 2} u_n \rightarrow | u |^{m - 2} u$ a.e.\ on $\mathbb{R}^N$. By a well-known remark of
{\citet[pp.\ 487]{MR699419}} (see \citet[Lem 1.4.8]{MR1276944} for a proof), we have $| u_n |^{m - 2} u_n
  \rightharpoonup | u |^{m - 2} u$ in $L^{m / (m - 1)}$. Hence
  \begin{equation}
    \int | u_n |^{m - 2} u_n v \rightarrow \int | u |^{m - 2} u v \text{\qquad
    for $v \in X$.} \label{he5}
  \end{equation}
  Similarly, given $v \in X$, $\{ | u_n |^q \}$ and $\big\{ \left| {u_n} 
  \right|^{q - 2} u_n v \big\}$ are bounded in $L^{2_{\ast} / q}$ and
  converge pointwise on $\mathbb{R}^N$. Employing the above Br{\'e}zis--Lieb remark again we have
  \[ | u_n |^q \rightharpoonup | u |^q \text{\quad in $L^{2_{\ast} / q}
     \text{, \qquad}$} | u_n |^{q - 2} u_n v \rightharpoonup | u |^{q - 2} u v
     \text{\quad in $L^{2_{\ast} / q}$.} \]
  By our assumption $g \in L^{2_{\ast} / (2_{\ast} - q)}$, the dual space of
  $L^{2_{\ast} / q}$, we get
  \begin{equation}
    \int g | u_n |^q \rightarrow \int g | u |^q \text{, \qquad} \int g | u_n
    |^{q - 2} u_n v \rightarrow \int g | u |^{q - 2} u v \text{.} \label{he6}
  \end{equation}
  By Lemma \ref{9},
  \begin{align}
    \lim_{n \rightarrow \infty} \int | \nabla u_n |^{p - 2} \nabla u_n \cdot
    \nabla v & =  \lim_{n \rightarrow \infty} \langle \nabla \mathcal{N}
    (u_n), v \rangle \nonumber\\
    & =  \langle \nabla \mathcal{N} (u), v \rangle = \int | \nabla u |^{p -
    2} \nabla u \cdot \nabla v \text{,} \label{he7} \\
    \varliminf_{n \rightarrow \infty} \frac{1}{p} \int | \nabla u_n |^p & = 
    \varliminf_{n \rightarrow \infty} \mathcal{N} (u_n) \geq \mathcal{N} (u) =
    \frac{1}{p} \int | \nabla u |^p \text{.} \label{he8} 
  \end{align}
  From (\ref{he5}), (\ref{he6}) and (\ref{he7}) we see that for $v \in X$
  \begin{align*}
    0 & =  \lim_{n \rightarrow \infty} \langle \nabla \Phi (u_n), v \rangle\\
    & =  \lim_{n \rightarrow \infty} \left( \langle u_n, v \rangle + \int |
    \nabla u_n |^{p - 2} \nabla u_n \cdot \nabla v - \lambda \int g | u_n
    |^{q - 2} u_n v - \int | u_n |^{m - 2} u_n v \right)\\
    & =  \langle u, v \rangle + \int | \nabla u |^{p - 2} \nabla u \cdot
    \nabla v - \lambda \int g | u |^{q - 2} u v - \int | u |^{m - 2} u v\\
    & =  \langle \nabla \Phi (u), v \rangle \text{.}
  \end{align*}
  It follows that the weak limit $u$ is a critical point of $\Phi$ at
  nonnegative level:
  \begin{align}
    \Phi (u) & =  \Phi (u) - \frac{1}{q} \langle \nabla \Phi (u), u \rangle
    \nonumber\\
    & =  \left( \frac{1}{2} - \frac{1}{q} \right) \| u \|^2 +\left( \frac{1}{p} - \frac{1}{q} \right)\int|\nabla u|^p + \left( \frac{1}{q} -
    \frac{1}{m} \right) \int | u |^m \geq 0 \text{.} \label{he1} 
  \end{align}
  Let $v_n = u_n - u$. By {\citet[Thm 1]{MR699419}}
  and the fact that $X$ is a Hilbert space, we have
  \begin{align}
    \| u_n \|^2 & =  \| u \|^2 + \| v_n \|^2 + o (1) \text{,} \label{hB} \\
    \int | u_n |^m & =  \int | u |^m + \int | v_n |^m + o (1) \text{.}
    \label{hA} 
  \end{align}
  Combining the first limit in (\ref{he6}) we deduce
  \begin{align}
    \Phi (u_n) & =  \frac{1}{2} \| u_n \|^2 + \frac{1}{p} \int | \nabla u_n
    |^p - \frac{\lambda}{q} \int g | u_n |^q - \frac{1}{m} \int | u_n |^m \nonumber\\
    & =  \frac{1}{2} \| u \|^2 + \frac{1}{2} \| v_n \|^2 + \frac{1}{p} \int
    | \nabla u_n |^p - \frac{\lambda}{q} \int g | u |^q \nonumber\\
    &   \qquad - \frac{1}{m} \int | u |^m - \frac{1}{m} \int | v_n |^m + o
    (1) \nonumber\\
    & =  \Phi (u) -\mathcal{N} (u) +\mathcal{N} (u_n) + \frac{1}{2} \| v_n
    \|^2 - \frac{1}{m} \int | v_n |^m + o (1) \text{,} \label{he3} 
  \end{align}
  and
  \begin{align}
    0 & =  \frac{1}{p} \langle \nabla \Phi (u_n), u_n \rangle + o (1)
    \nonumber\\
    & =  \frac{1}{p} \left( \| u_n \|^2 + \int | \nabla u_n |^p - \lambda
    \int g | u_n |^q - \int | u_n |^m \right) + o (1) \nonumber\\
    & =  \frac{1}{p} \left( \| u \|^2 + \| v_n \|^2 + \int | \nabla u_n |^p
    - \lambda \int g | u |^q - \int | u |^m - \int | v_n |^m \right) + o (1)
    \nonumber\\
    & =  \frac{1}{p} \left( \langle \nabla \Phi (u), u \rangle - \int |
    \nabla u |^p \right) + \frac{1}{p} \left( \| v_n \|^2 + \int | \nabla u_n
    |^p - \int | v_n |^m \right) + o (1) \nonumber\\
    & =  \frac{1}{p} \| v_n \|^2 -\mathcal{N} (u) +\mathcal{N} (u_n) -
    \frac{1}{p} \int | v_n |^m + o (1) \text{.} \label{he4} 
  \end{align}
  Substracting (\ref{he4}) from (\ref{he3}), using $\Phi (u_n) \rightarrow c$ we
  get
  \[ \Phi (u) + \left( \frac{1}{2} - \frac{1}{p} \right) \| v_n \|^2 + \left(
     \frac{1}{p} - \frac{1}{m} \right) \int | v_n |^m \rightarrow c \text{.}
  \]
  Up to a subsequence we may assume $\| v_n \|^2 \rightarrow a$, $| v_n |_m^m
  \rightarrow b$. Since $\Phi (u) \geq 0$ we deduce
  \begin{equation}
    c \geq \left( \frac{1}{2} - \frac{1}{p} \right) a + \left( \frac{1}{p} -
    \frac{1}{m} \right) b \text{.} \label{he0}
  \end{equation}
  Using (\ref{he6}), (\ref{he8}), (\ref{hB}) and (\ref{hA}) we get
  \begin{align*}
    0 & =  \varliminf_{n \rightarrow \infty} \langle \nabla \Phi (u_n), u_n
    \rangle\\
    & =  \varliminf_{n \rightarrow \infty} \left( \| u_n \|^2 + \int | \nabla
    u_n |^p - \lambda \int g | u_n |^q - \int | u_n |^m \right)\\
    & \geq  \| u \|^2 + a + \int | \nabla u |^p - \lambda \int g | u |^q -
    \int | u |^m - b\\
    & =  \langle \nabla \Phi (u), u \rangle + a - b = a - b \text{.}
  \end{align*}
  It follows that $b \geq a$. By definition of $S$, $\| v_n \|^2 \geq S | v_n
  |_m^2$, so
  \[ a \geq S b^{2 / m} \geq S a^{2 / m} \]
  Consequently, we have either $a = 0$ or $a \geq S^{m / (m - 2)}$.
  
  If $a \geq S^{m / (m - 2)}$, by (\ref{he0}) and $b \geq a$ we get
  \[ c \geq \left( \frac{1}{2} - \frac{1}{m} \right) a \geq \left( \frac{1}{2}
     - \frac{1}{m} \right) S^{m / (m - 2)} = c^{\ast} \text{,} \]
  a contradiction. Therefore, we must have $a = 0$. That is, $\| u_n - u \|
  \rightarrow 0$.
\end{proof}

\begin{remark}
  Unlike the classical problem
  \[ - \Delta u + \lambda u = | u |^{2^{\ast} - 2} u \text{, \qquad$u \in
     H_0^1 (\Omega)$} \]
  treated for example in {\citet[{\textsection}1.10]{MR1400007}}, due to the
  presence of $\mathcal{N} (u)$, which is only weakly lower semi-continuous
  (Lemma \ref{9}), we can only derive $b \geq a$, which can not be used
  together with the direct consequence of (\ref{he3})
  \[ c \geq \frac{1}{2} a - \frac{1}{m} b \]
  to produce a contradiction. By substracting (\ref{he4}) from (\ref{he3}), the
  bad terms involving $\mathcal{N}$ cancel out, yielding (\ref{he0}) in which
  $b$ has positive coefficient, thus resolve this problem. This idea was first used in \cite{MR5057254} for the study of {S}chr\"odinger-{P}oisson-{S}later
 equations.
\end{remark}

Having verified the local $(\tmop{PS})$ condition, we are ready to present the
proof of Theorem \ref{ht5}. For a symmetric subset $A \subset X \backslash \{ 0
\}$ we denote by $i (A)$ the cohomological index of $A$, introduced by  {\cite{MR0478189}}. If $A$ is homeomorphic to the unit sphere
$S^{d - 1}$ in $\mathbb{R}^d$, then $i (A) = d$.

\begin{proposition}[{\citet[Thm 2.1]{MR4999814}}]
  \label{hP}Let $X$ be a Banach space, $\Phi : X \rightarrow \mathbb{R}$ be an
  even $C^1$-functional satisfying $(\tmop{PS})_c$ for $c \in (0, c^{\ast})$
  being $c^{\ast}$ some positive constant. If $0$ is a strict local minimizer
  of $\Phi$ and there are $R > 0$ and a compact symmetric set $A \subset
  \partial B_R$, such that $i (A) = k$,
  \begin{equation}
    \max_{u \in A} \Phi (u) \leq 0 \text{, \qquad} \max_{(t, u) \in [0, 1]
    \times A} \Phi (t u) < c^{\ast} \text{,} \label{h3e}
  \end{equation}
  then $\Phi$ has $k$ pairs of nonzero critical points with positive critical
  values.
\end{proposition}

\begin{proof}[Proof of Theorem \ref{ht5}]
  By Lemma \ref{h3l}, $\Phi$ satisfies $(\tmop{PS})_c$ for $c \in (0,
  c^{\ast})$. In both cases $2 < p \le q < m$ and $2 = p < q < m$, by the expression of $\Phi$ it is
  clear that $u = 0$ is a strict local minimizer of $\Phi$.
  
  Given $k \in \mathbb{N}$, let $Z = \{ u \in X \mid \tmop{supp} u \subset
  \Omega \}$ and $Z_k$ be an $k$-dimensional subspace of $Z$. Because $g > 0$
  on $\Omega$,
  \[ [u] = \left( \frac{1}{q} \int g | u |^q \right)^{1 / q} \]
  is a norm on $Z_k$. Since all norms on the finite dimensional space $Z_k$
  are equivalent,
  \begin{align}
    \Phi (u) & =  \frac{1}{2} \| u \|^2 + \frac{1}{p} | \nabla u |_p^p -
    \lambda [u]^q - \frac{1}{m} | u |_m^m \nonumber\\
    & \leq  c_1 \| u \|^2 + c_2 \| u \|^p - \lambda c_3 \| u \|^q - c_4 \| u
    \|^m \nonumber\\
    & =  f (\| u \|) - \lambda c_3 \| u \|^q \text{,} \label{hqe} 
  \end{align}
  where $f (r) = c_1 r^2 + c_2 r^p - c_4 r^m$. Set $A = Z_k \cap \partial B_R$
  for $R > 1$ satisfying $f (R) < 0$. Then $i (A) = k$, $\max_A \Phi < 0$. So $\Phi$ satisfies the first condition in (\ref{h3e}).
  
  Since $f (0) = 0$, there is $\delta > 0$ such that $f (t R) < c^{\ast}$ for
  $t \in [0, \delta]$. Thus by (\ref{hqe})
  \begin{equation}
    \max_{(t, u) \in [0, \delta] \times A} \Phi (t u) \leq \max_{t \in [0,
    \delta]} f (t R) < c^{\ast} \text{.} \label{hf1}
  \end{equation}
  It is clear that there is $\lambda_0 > 0$ such that for $\lambda >
  \lambda_0$ there holds
  \begin{align}
    \max_{(t, u) \in [\delta, 1] \times A} \Phi (t u) & \leq  \max_{t \in
    [\delta, 1]} \{ f (t R) - \lambda c_3 t^q R^q \} \nonumber\\
    & \leq  \max_{t \in [\delta, 1]} f (t R) - \lambda c_3 \delta^q R^q
    \text{} < c^{\ast} \text{.} \label{hf2} 
  \end{align}
  From (\ref{hf1}) and (\ref{hf2}), we see that if $\lambda > \lambda_0$ then
  $\Phi$ satisfies the second condition in (\ref{h3e}). Applying Proposition \ref{hP}, $\Phi$ has $k$ pairs of critical points. So (\ref{he2}) has $k$
  pairs of solutions.
\end{proof}

\subsection*{Acknowledgment}
The authors are very grateful to the reviewer for carefully reading the manuscript and providing very valuable comments and suggestions.


\end{document}